\documentclass[11pt,reqno]{amsart}
\usepackage{amsthm, amssymb, amsmath,graphicx,multirow,psfrag}
\newcommand{\lk}{\ell k}
\newcommand{\Z}{\mathbb{Z}}
\newcommand{\ML}{\mathcal{L}}

\newcommand{\LE}{\lambda_{L}}
\newcommand{\I}{\chi(I_*\,}
\newcommand{\SL}{$s(\ell_1 \cup \ell_2 \subset \Sigma)$ }
\newcommand{\ad}{\operatorname{ad}}

\newcommand{\tr}{\operatorname{tr}}

\newcommand{\Stab}{\operatorname{Stab}}

\newcommand{\Tor}{\operatorname{Tor}}
\newcommand{\In}{I^{\natural}_{*}}

\newcommand{\pt}{\operatorname{point}}
\newcommand{\Br}{\operatorname{Br}}

\newtheorem{theorem}{Theorem}[section]
\newtheorem*{theorem1}{Theorem}
\newtheorem{prop}[theorem]{Proposition}
\newtheorem{lemma}[theorem]{Lemma}

\theoremstyle{definition}
\newtheorem*{remark}{Remark}

\newtheorem*{example}{Example}

\title{Lescop's Invariant and Gauge Theory}
\thanks{The author was partially supported by NSF Grant 1065905}
\author{Prayat Poudel}
\address{Department of Mathematics \newline \indent University of Miami \newline \indent P.O.Box 249085 \newline \indent Coral Gables, FL 33124}
\email{\rm{poudelp@math.miami.edu}}

\begin{document}

\begin{abstract}
Taubes proved that the Casson invariant of an integral homology $3$-sphere equals half the Euler characteristic of its instanton Floer homology. We extend this result to all closed oriented $3$-manifolds with positive first Betti number by establishing a similar relationship between the Lescop invariant of the manifold and its instanton Floer homology. The proof uses surgery techniques.
\end{abstract}

\maketitle
\section{Introduction}
The Lescop invariant $\LE(M)$ is a rational valued invariant of closed 
oriented $3$-manifolds $M$ defined by Lescop \cite{LE} via a combinatorial 
formula as a generalization of the Casson invariant \cite{AM}. The Casson
invariant, while only defined for integral homology $3$-spheres, has a very 
useful gauge theoretic interpretation, due to Taubes \cite{CT}, as half the 
Euler characteristic of the instanton Floer homology \cite{F1}. We provide 
a similar interpretation of the Lescop invariant for all $3$-manifolds with the 
positive first Betti number using a version of instanton Floer homology 
based on admissible bundles. In fact, our formula matches the one 
conjectured in the physics literature, where the Lescop invariant arises 
as a partition function of the Donaldson-Witten theory of a 4-manifold of 
the form $S^1 \times M$, see Mari\~no--Moore \cite{MM}.

\begin{theorem1} {\label{Thm:First}}
Let $M$ be a closed oriented connected $3$-manifold with $b_1(M) \ge 1$, and 
let $\LE(M)$ be its Lescop invariant. Then there exists an admissible bundle 
$P$ over $M$ such that
\begin{align*}
\LE(M) & = -\displaystyle\frac{1}{2}\, \I(M,P)) - \displaystyle\frac{1}{12}\,|\Tor(H_1(M))|, \quad \textnormal{if } b_1(M) =1,\; \textnormal{ and} \\
\LE(M) & = \displaystyle\frac{1}{2}\, (-1)^{b_1(M)} \cdot |\Tor(H_1(M))| \cdot \I(M,P)), \quad \textnormal{if } b_1(M) \geq 2,
\end{align*}
where $\I(M,P))$ stands for the Euler characteristic of the instanton Floer homology of the pair $(M,P)$, see Section \ref{inst}.
\end{theorem1}

In addition, if $H_1(M)$ has no $2$-torsion, we show that $\I(M,P))$ is independent of the choice of admissible bundle $P$ hence the above formulas hold for \emph{any} admissible bundle $P$. We conjecture that the same is true for any $M$ with $b_1(M) \geq 1$, with or without $2$-torsion in its homology.

Still lacking is a gauge theoretic interpretation of the Lescop invariant 
for rational homology 3-spheres with non-trivial torsion because there is 
no satisfactory definition of instanton Floer homology for such manifolds.

Our proof will proceed by induction on the first Betti number $b_1 (M)$ of the 
manifold and will use the Floer exact triangle \cite{F2}. In case $M$ has no 
torsion in homology, we will start the induction at $b_1 (M) = 0$ and use 
Taubes' theorem \cite{CT}. In the presence of torsion, due to the 
aforementioned problem with defining instanton Floer homology for rational 
homology spheres, we will start at $b_1 (M) = 1$ and use an extension of 
Taubes' theorem due to Masataka \cite{MT}. 

The paper also contains applications of our main theorem to the singular instanton knot homology of Kronheimer and Mrowka \cite{KMK} and to the instanton homology of two-component links of Harper and Saveliev \cite{HS}, and an example explaining the factor $|\Tor(H_1(M))|$ in the Lescop invariant from a gauge-theoretic viewpoint.

\subsection*{Acknowledgments}
I am immensely grateful to my Ph.D. advisor Dr. Nikolai Saveliev for his guidance and support. I am also thankful to Dr. Ken Baker and Dr. Kim Fr\o yshov for their helpful suggestions. 


\section{Instanton Floer Homology}{\label{inst}}
Let $M$ be a closed oriented connected $3$-manifold and $P \rightarrow M$ a 
$U(2)$-bundle such that one of the following conditions holds:
\begin{enumerate}
\item[(1)] $M$ is an integral homology sphere and $P$ is a trivial ${SU}(2)$-bundle, or
\item[(2)] $b_1(M) \geq 1$ and $P$ is a $U(2)$-bundle whose first Chern class $c_1(P)$ has an odd pairing with some integral homology class in $H_2(M)$. 
Note that the second Stiefel-Whitney class $w_2(\ad(P)) \in H^{2}(M;\Z_2)$ of the associated $SO(3)$-bundle $\ad(P)$ is then not zero as a map $H_2(M) \rightarrow \Z_2$.
\end{enumerate}

Both the bundle $P$ and its adjoint bundle $\ad(P)$ will be referred to as admissible bundles.

Given an admissible bundle $P$, consider the space $\mathcal{C}$ of $SO(3)$-connections in $\ad(P)$. This space is acted upon by the group $\mathcal{G}$ of determinant one gauge transformations of $P$. The instanton Floer homology $I_*(M,P)$ is the Floer homology arising from the Chern-Simons function of the space $\mathcal{C} / \mathcal{G}$, see Donaldson \cite{DO}.

The groups $I_*(M,P)$ depend on the choice of $\ad(P)$ but not $P$. They have an absolute $\Z_2$-grading defined as in Section 5.6 of \cite{DO}. These groups also admit a relative $\Z_8$-grading, which becomes an absolute $\Z_8$-grading if $M$ is an integral homology sphere. Note that our setup is consistent with that of Kronheimer and Mrowka \cite{KMS}, Section $7.1$. Using their notations, 
$I_{*}(M,P) = I_{*}(M)_w$, where $w = \det\,(P)$ is the determinant bundle of $P$.

Two main ingredients that go into our calculation of the Euler characteristic of $I_{*}(M,P)$ are as follows. The first one is a special surgery presentation of $M$ as in \cite{LE}, Lemma 5.1.1. 

\begin{lemma}{\label{Lem:Les}}
Let M be a closed connected oriented 3-manifold and $b_1(M) = n \geq 0$. 
Then  there exists a rational homology sphere $\Sigma$ and an algebraically  
split n-component link $\ML  \subset \Sigma$ such that $M = \Sigma  + 0 \cdot \ML$, each component of $\ML$ is null-homologous in $\Sigma$, and $|H_1(\Sigma)| = |\Tor(H_1(M))|$.
\end{lemma}

We denote by $\Sigma + 0 \cdot \ML$ the manifold obtained by 0-framed surgery on $\ML$. Similarly, $\Sigma \pm \ell$ will denote the result of $(\pm 1)$-surgery on knot $\ell$.

The second ingredient is a long exact sequence known as the Floer exact triangle, see \cite{BD}. Let $\Sigma$ be a rational homology sphere and $\ML = \ell_1 \cup \ell_2 \ldots \cup \ell_n$ be an algebraically split link in $\Sigma$. If $n=1$, we will require that $\Sigma$ be an integral homology sphere. Then we have the following Floer exact triangle\,:

\medskip

\begin{picture}(220,68)
	\put (88,50)	{$I_*(\Sigma + 0 \cdot \ell_{1}+ \ldots + 0\cdot \ell_n)$} 
	\put (127,40)	{\vector (-1,-1){20}}
	\put (95,32)	{$\empty_{-1}$}
	\put (207,20)	{\vector (-1,1){20}}
	\put (-10,2)	{$I_*\, (\Sigma + 0 \cdot \ell_{1}+ \ldots + 0\cdot \ell_{n-1} )$} 
	\put (143,5)	{\vector (1,0) {30}}
	\put (180,2)	{$I_*\,((\Sigma+ 0 \cdot \ell_1 + \ldots 0 \cdot \ell_{n-1})  - \ell_{n})$}
\end {picture}
\vspace*{4mm}

\noindent
The admissible bundles $P$, which are omitted from the notations, have $c_1(P)$ pair non-trivially with the natural homology classes obtained by capping off a Seifert surface of $\ell_j$ by a meridional disk of the surgery. In addition, the three admissible bundles that show up in the Floer exact triangle match when restricted to the exterior of the link $\ell_1 \cup \cdots \cup \ell_n$ in $\Sigma$. If $n=1$, the bundles $P$ over $\Sigma$ and $ \Sigma - \ell_1$ are trivial $SU(2)$-bundles. 

It should be pointed out that we will not use the complete strength of the Floer exact triangle; all we will derive from it is the following relation on Euler characteristics:
\begin{align*}
&\I (\Sigma + 0 \cdot \ell_1 + 0 \cdot \ell_2 + \ldots + 0 \cdot \ell_n)) = \\
& \I (\Sigma+ 0 \cdot \ell_1 + \ldots + 0 \cdot \ell_{n-1})-\ell_n)
 - \I (\Sigma + 0 \cdot \ell_1 + \ldots+ 0 \cdot \ell_{n-1})). 
\end{align*}
The obvious observation that 
\begin{align*}
& b_1((\Sigma+ 0 \cdot \ell_1 + \ldots + 0 \cdot \ell_{n-1})-\ell_n)) =b_1(\Sigma + 0 \cdot \ell_1 + \ldots+ 0 \cdot \ell_{n-1})\\
& = b_1(\Sigma + 0 \cdot \ell_1 + 0 \cdot \ell_2 + \ldots 0 \cdot \ell_n) - 1
\end{align*}
allows us to proceed via induction on $b_1(M)$. 
\medskip

\section{Case of $b_1(M) =0$}
If $M$ is an integral homology sphere, then $\LE(M) = \lambda(M)$, which is the Casson invariant of $M$, see Section 1.5 of \cite{LE}. On the Floer homology side, we work with the trivial $SU(2)$-bundle $P$ over $M$, and denote the instanton Floer homology by $I_{*}(M,P)$. According to Taubes \cite{CT}, we have $\I(M,P)) = 2 \cdot \lambda(M)$. Therefore, $\I(M,P)) = 2 \cdot \LE(M)$.


\section{Case of $b_1(M) = 1$}
Let $k \subset \Sigma$ be a null-homologous knot in a rational homology sphere $\Sigma$. Choose a Seifert surface $F$ of $k$, and denote by $V$ its Seifert matrix with respect to a basis of $H_1(F)$. The Laurent polynomial 
\[
\Delta_{k \subset \Sigma}(t) = |H_1(\Sigma)|\,\cdot\, \det\,(t^{1/2}\,V - t^{-1/2}\,V^{\top})
\]
is called the Alexander polynomial of $k \subset \Sigma$. Note that $\Delta_{k\subset \Sigma}(t) = \Delta_{k\subset \Sigma}(-t)$ and $\Delta_{k\subset \Sigma}(1) = |H_1(\Sigma)| > 0$.

Given a closed oriented connected $3$-manifold $M$ with $b_1(M) = 1$, according to Lemma \ref{Lem:Les}, there exists a null-homologous knot $\ell_1 \subset \Sigma$ in a rational homology sphere $\Sigma$ such that $M = \Sigma + 0 \cdot \ell_1$. The Lescop invariant of $M$ is then equal to 
\begin{equation}{\label{eq:LEb1}}
\LE(M) = \frac{1}{2}{\Delta^{''}_{\ell_1 \subset \Sigma}}(1) - \frac{1}{12} \; |\Tor(H_1(M))|,
\end{equation}
see \cite{LE}, Section 1.5. It is independent of the choice of surgery presentation of $M$. 

In the special case when $\Sigma$ is an integral homology sphere, it follows from the Floer exact triangle and Casson's surgery formula \cite{AM} that 
\begin{align*}
	\I(M,P)) &= \I(\Sigma - \ell_{1})) - \I(\Sigma))= 2 \lambda(\Sigma - \ell_{1}) - 2\lambda(\Sigma) \\
	&= 2\left(-\frac{1}{2} {\Delta^{''}_{\ell_1 \subset \Sigma}}(1)\right)= -{\Delta^{''}_{\ell_1 						\subset \Sigma}}(1).
\end{align*}Therefore, 
\begin{equation}
\I(M,P)) = - 2 \LE(M) - \displaystyle\frac{1}{6} \; |{\Tor} (H_1(M))|
\end{equation}
as claimed, for the unique admissible $SU(2)$-bundle $P$ over $M$. The general case is handled similarly using the following result.

\begin{prop}
Let $M$ = $\Sigma + 0 \cdot \ell_1$, where $\ell_1$ is a null-homologous knot in a rational homology sphere $\Sigma$. Then $\I(M,P)) = -\Delta_{\ell_1 \subset \Sigma}^{''}(1)$ for any admissible bundle $\ad(P)$ over $M$ whose restriction to the exterior of $\ell_1 \subset \Sigma$ is trivial. 
\end{prop}

\begin{proof}
If $H_1(\Sigma)$ has non-trivial torsion, the starting point for our calculation will be the result from \cite{MT} which, with our normalization, reads
\begin{equation*}
\frac{1}{2} \I(M,P)) =  - \LE(\Sigma+ \ell_1) + \LE(\Sigma).
\end{equation*}
We wish to identify the right hand side of this equation with $-\frac{1}{2} \Delta^{''}_{\ell_1 \subset \Sigma}(1)$.
By Lescop \cite{LE}, 
\begin {equation*}
	\LE(\Sigma + \ell_{1}) = \displaystyle \frac {\vline \, H_1(\Sigma+\ell_{1}) \, \vline}	{\vline \,	H_1(\Sigma)\, \vline} \cdot \LE(\Sigma) + \mathbb{F}_{\Sigma}	(\ell_{1}),\\*
\end{equation*}
where $\mathbb{F}_{\Sigma}(\ell_{1})$ is defined by equation 1.4.8 in \cite{LE}. Since $|H_1(\Sigma + \ell_1)| = |H_1(\Sigma)|$, we conclude that $\LE(\Sigma + \ell_{1}) = \LE(\Sigma) + \mathbb{F}_{\Sigma}(\ell_{1})$ and therefore 
\begin{equation*}
\frac{1}{2} \I(M,P))= -\mathbb{F}_{\Sigma}(\ell_{1}).
\end{equation*}
A straightforward calculation of $\mathbb{F}_{\Sigma}(\ell_{1})$ shows that
\begin{equation*}
\mathbb{F}_{\Sigma}(\ell_1) = \frac {1}{2} \Delta_{\ell_1 \subset \Sigma}^{''}(1),
\end{equation*}
which leads to the desired formula. Therefore, for our choice of admissible bundle $P$ over $M$, 
\begin{equation*}
\I(M,P)) = - 2 \LE(M) - \frac{1}{6} \; |{\Tor} (H_1(M))|.
\end{equation*}
\end{proof}

\begin{remark}
The exact sequence of the universal coefficient theorem, 
\begin{equation*}
0 \longrightarrow \operatorname{Ext}(H_1(M),\Z_2) \longrightarrow H^2(M;\Z_2) \longrightarrow \operatorname{Hom}(H_2(M), \Z_2) \longrightarrow 0 ,
\end{equation*}
tells us that the ambiguity in choosing an admissible bundle $\ad(P)$ over $M = \Sigma + 0 \cdot \ell_1$ resides in the group $\operatorname{Ext}(H_1(M),\Z_2)$. If $H_1(M)$ has no $2$-torsion, the latter group vanishes, implying that there is a unique admissible bundle $\ad(P)$ over $M$. 
\end{remark}


\section{Case of $b_1(M) = 2 $}
Let $\ML$ = $\ell_{1} \cup \ell_{2}$ be an oriented two-component link in a rational homology sphere $\Sigma$ such that $\ell_{1}$, $\ell_{2}$ are null-homologous in  $\Sigma$ and $\lk_{\Sigma}(\ell_{1}, \ell_{2})$ = 0. There exist Seifert surfaces $F_1$ and $F_2$ of $\ell_1$ and $\ell_2$, respectively, such that $F_1 \cap \ell_2$ = $\emptyset$ and $F_2 \cap \ell_1$ = $\emptyset$. If the surfaces intersect, they may be assumed to intersect in a circle $c$, see \cite{TC}. The self linking number of $c$ with respect to either $ F_{1} $ or $ F_{2} $ is called the Sato-Levine invariant and is denoted by \SL . To be precise, \SL = $\lk_{\Sigma}(c,c^{+})$, where $c^+$ is a positive push-off of $c$ with respect to either $F_1$ or $F_2$. If the surfaces don't intersect then \SL =$0$.

Given a closed oriented connected $3$-manifold $M$ with $b_1(M) = 2$, according to Lemma \ref{Lem:Les}, there exists an algebraically split link $\ell_1 \cup \ell_2 \subset \Sigma$ in a rational homology sphere $\Sigma$ such that $M = \Sigma + 0 \cdot \ell_1 + 0 \cdot \ell_2$. According to Lescop \cite{LE}, Section 5.1,
\begin{equation}{\label{eq:LEb2}} 
\LE(M) = -|\Tor(H_1(M))| \cdot s(\ell_1 \cup \ell_2 \subset \Sigma);
\end{equation}
it is independent of the choice of surgery presentation of $M$. We will first handle the case when $\Sigma$ is an integral homology sphere.

\begin{prop}\label{thm2.1}
Let $M = \Sigma + 0 \cdot \ell_{1} + 0 \cdot \ell_{2}$, where $\ell_1\cup \ell_2$ is a link in an integral homology sphere $\Sigma$ such that $\lk_{\Sigma}(\ell_1,\ell_2) = 0$. Then, for any admissible bundle $\ad(P)$ over $M$,
\begin{equation}{\label{eq:LEXb2}}
\I(M,P)) = 2\cdot \LE(M).
\end{equation}
\end{prop}

\begin{proof}
Let $P$ be a $U(2)$-bundle over $M$ such that $w_2(\ad(P))$ evaluates non-trivially on both homology classes obtained by capping off Seifert surfaces of $\ell_1$ and $\ell_2$ by meridional disks of the surgery. We will indicate this by writing $w_2(\ad(P)) = (1,1)$ in the natural basis of $\operatorname{Hom}(H_2(M),\Z_2)$. It follows from the Floer exact triangle that
\begin{align*}
	\I(M,P))
	&= \I((\Sigma+ 0 \cdot \ell_{1}) - \ell_{2}))- \I(\Sigma + 0 \cdot \ell_{1}))\\
	&= -{\Delta^{''}_{\ell_1 \subset \Sigma - \ell_2}(1)} + {\Delta^{''}_{\ell_1 \subset 						\Sigma}}(1),
\end{align*}
hence the calculation boils down to computing Alexander polynomials of $\ell_1 \subset \Sigma$ and $\ell_1 \subset \Sigma - \ell_2$.

According to Hoste \cite{HC}, there exist Seifert surfaces $F_{1}$ and $F_{2}$ of $\ell_{1}$ and $\ell_{2}$, respectively, such that $F_{1}\,\cap\,F_{2}$ is either empty or a single ribbon intersection, and furthermore, $F_{1}\,\cap\,F_{2}$ does not separate $F_{1}$ or $F_{2}$.

If $F_{1}\,\cap\,F_{2}$ is empty then $\ell_1 \cup \ell_2$ is a boundary link and \SL = $0$. Moreover, $(-1)$-surgery along $\ell_2$ leaves $F_1$ unaffected, therefore, ${\Delta_{\ell_1 \subset\Sigma}}(t)$ = $\Delta_{\ell_1 \subset \Sigma - \ell_{2}}(t)$ and $\I(M,P))= 0$ as desired. 

\bigskip\medskip

\begin{figure}[!ht]
\begin{minipage}{0.45\linewidth}
\centering
\psfrag{F1}{$F_1$}
\psfrag{F2}{$F_2$}
\includegraphics[width=\textwidth]{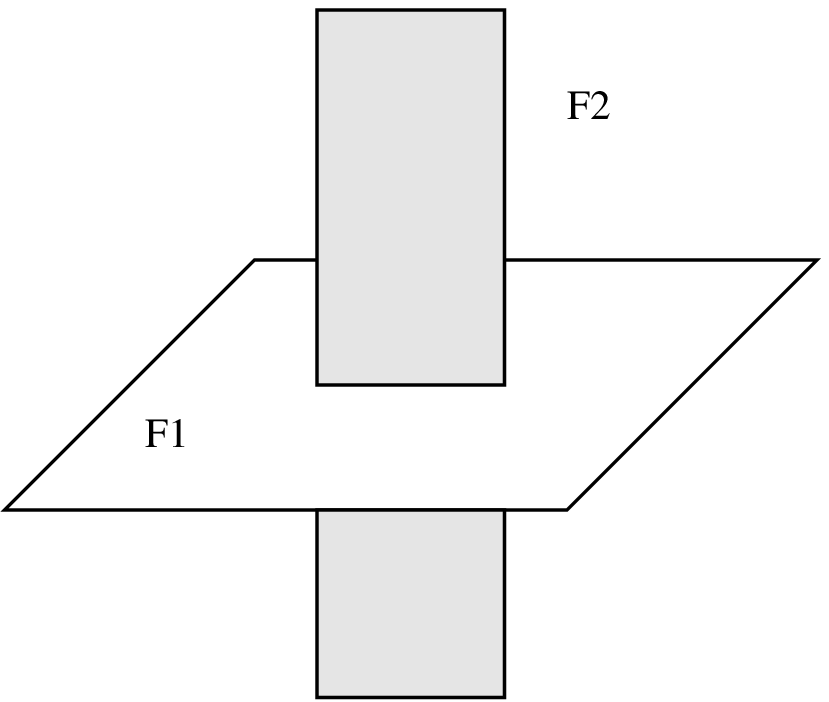}
\caption{}\label{Fig1}
\end{minipage}
\hspace{0.5cm}
\begin{minipage}{0.45\linewidth}
\centering
\psfrag{F1}{$F'_1$}
\psfrag{F2}{$F_2$}
\psfrag{e1}{$e_1$}
\includegraphics[width=\textwidth]{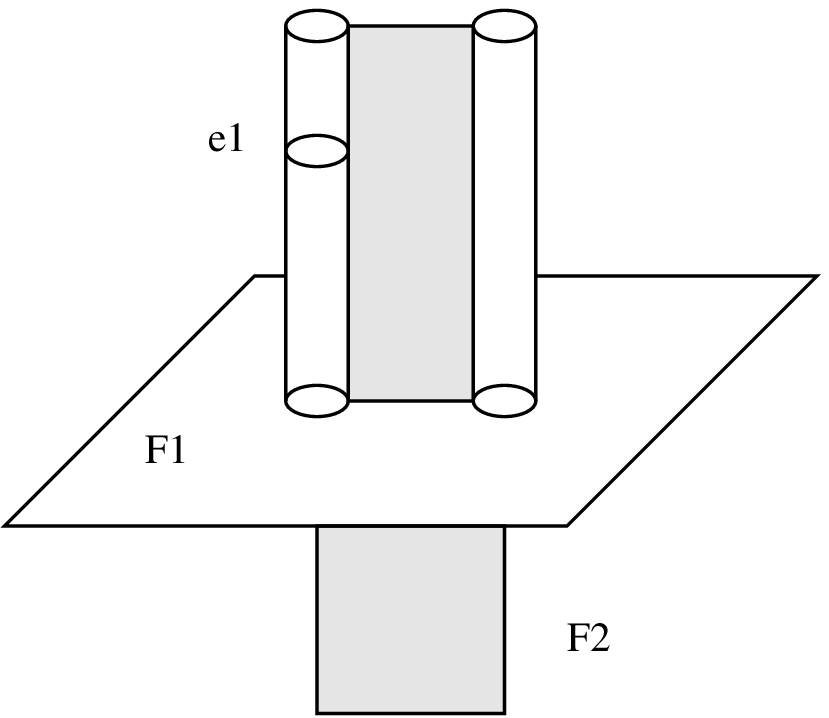}
\caption{}\label{Fig2}
\end{minipage}
\end{figure}
\medskip

If $F_{1}\,\cap\,F_{2}$ is non-empty then it is a single ribbon intersection as in Figure ~\ref{Fig1}. The intersection $F_1 \cap \partial F_2 \subset \ell_2$ consists of two points that separate $\ell_2$ into two arcs, $\ell^{'}_2$ and $\ell^{''}_2$. Stabilize surface $F_1$ by adding a tube with core $\ell^{'}_2$ and call this new modified surface $F^{'}_1$, see Figure ~\ref{Fig2}. If $F_1$ had genus $g$, the genus of $F^{'}_1$ will be $g+1$.

The intersection $F_{1}^{'} \cap F_{2}$ is a closed loop $c$ which represents a primitive homology class in $H_1(F^{'}_1)$. Complete $c$ to a basis $\{ e_{1}, c, e_{3},...,e_{2g+2} \}$ of $H_{1}(F_{1}^{'})$,  where $e_1$ is a meridional curve of $\ell_2$ and $\{ e_{3},...,e_{2g+2}\}$ is a basis for  $H_{1}(F_{1})$. Using general position argument we will assume that ${e_m} \cap F_{2} = \emptyset$ and hence $\lk_{\Sigma}(e_m, \ell_2)$ = $0$ for $m \geq 3$. In addition, it is obvious that $\lk_{\Sigma}(e_1, \ell_2) = \pm 1$, and that $\lk_{\Sigma}(c,\ell_2)$ = $\lk_{\Sigma}(c^+,\ell_2)$ = $0$, where $c^+$ is a positive push-off of $c$ with respect to $F_2$. To summarize, we have the matrix

\begin{equation*}
	\label{E}	
	E \;=
	\begin{bmatrix}	
	\lk_{\Sigma}(e_{1},\ell_{2}) \\
	\lk_{\Sigma}(c,\ell_{2})\\
	\vdots \\
	\lk_{\Sigma}(e_{2g+2},\ell_{2})\\
	\end{bmatrix}
	=
	\begin{bmatrix}
		\pm 1\\
		0\\
		\vdots\\
		0\\
	\end{bmatrix}.
\end{equation*}

\medskip\noindent
This matrix accounts for the difference in the Seifert matrices of $\ell_1$ when viewed as a knot in $\Sigma$ and $\Sigma - \ell_2$. To be precise, according to Hoste \cite{HC},
\begin{equation}{\label{eqn:Hoste}}
 V(\ell_{1} \subset\Sigma - \ell_2) = V(\ell_{1} \subset\Sigma) + E \cdot E^\top.
\end{equation}

We would like next to determine the Seifert matrix $V(\ell_{1} \subset \Sigma)$ with respect to the basis $\{e_1,c, e_3, \ldots e_{2g+2}\}$ in $H_1(F^{'}_1)$. To begin, $\lk_{\Sigma}(e_{1},e_{1}^{+})$ = $0$ since the push off $e_1^+$ is disjoint from the meridional disk $D$ of $e_1$. For $m \geq 3$, the curve $e_m$ lies on the surface $F_1$. Therefore, its push-off $e_m^{+}$ can be isotoped to make it disjoint from $D$ so $\lk_{\Sigma}(e_1,e_m^{+}) =0$. Similarly $\lk_{\Sigma}(e_m, e_1^{+}) = 0$ for $m \geq 3$. Next, $\lk_{\Sigma}(c, e_1^{+})$ = $\pm 1$ and since $c^+$ does not intersect the meridional disk bounded by $e_1$, $\lk_{\Sigma}(e_1,c^{+})$ = 0. To finish the calculation, observe that $V(\ell_{1} \subset \Sigma) - V(\ell_{1} \subset \Sigma)^{\top} =  - I$, where $I: H_1(F_{1}^{'}) \times H_1(F_{1}^{'}) \rightarrow \Z$ denotes the intersection form given by $I(v,w)$ = $v \cdot w$. By general position argument, $c$ $\cdot$ $e_m$ = $0$ hence $\lk_{\Sigma}(c^{+}, e_m)- \lk_{\Sigma}(c,e_{m}^{+}) = 0$. For $m\geq 3$, denote $\lk_{\Sigma}(c^{+}, e_m)= \lk_{\Sigma}(c,e_{m}^{+})= a_{m-2}$. Finally, $\lk_{\Sigma}(c,c^{+})$= $s$ by the definition of $s$ = $s(\ell_1 \cup \ell_2 \subset \Sigma)$. Therefore, we obtain the matrix

\begin{center}
	$V(\ell_1 \subset \Sigma)=$
	$\left[ 
	\begin{array}{cc|cccc}
	0 & 0 &  & 0 &  \cdots & 0 \\
	\pm 1 & s &  & a_1 & \cdots &a_{2g}\\ \hline
   0 & a_1 & & \multicolumn{3}{c}{\multirow{3}{*}{\raisebox{-7mm}{\scalebox{1}{$W$}}}} \\
   \raisebox{2pt}{\vdots} & \vdots & & & &\\
   0 &a_{2g} & & & &
	\end{array}
	\right]$
\end{center}
where $W$ is the Seifert matrix of $\ell_1$ with respect to the basis $e_3, \ldots e_{2g} \in H_{1}(F_1)$.
Using Hoste's formula (\ref{eqn:Hoste}), we obtain
\begin{center}
	$V(\ell_1 \subset \Sigma - \ell_2)$ =
	$\left[ 
	\begin{array}{cc|cccc}
	1 & 0 & & 0 &  \cdots & 0 \\
	\pm 1 & s &  & a_1 & \cdots &a_{2g}\\ \hline
   0  & a_1 & & \multicolumn{3}{c}{\multirow{3}{*}{\raisebox{-7mm}{\scalebox{1}	{$W$}}}} 	\\
   \raisebox{2pt}{\vdots} & \vdots & & & &\\
   0 &a_{2g} & & & 
	\end{array}
	\right]$
\end{center}
Now we are ready to compute the Alexander polynomials. Denote $V(\ell_1 \subset \Sigma - \ell_2)$ by $V$ and let $z = t^{1/2} - t^{-1/2}$. Then
\begin{center}
	$ t^{1/2}V - t^{-1/2}V^{\top}=
	\left[ 
	\begin{array}{cc|cccc}
	z & \mp t^{-1/2} & & 0 &  \cdots & 0 \\
	\pm t^{1/2} & sz & & a_1z & \cdots &a_{2g}z\\ \hline
   0  & a_1z & & \multicolumn{3}{c}{\multirow{3}{*}{\raisebox{-7mm}{\scalebox{1}					{$t^{1/2}W - t^{-1/2}W^{\top}$}}}} \\
   \raisebox{2pt}{\vdots} & \vdots & & & &\\
   0 &a_{2g}z & & & & 
	\end{array}
	\right]$
\end{center}
Hence
$\det(t^{1/2}V - t^{-1/2}V^{\top})=	$

\begin{align*}
	&= \:z
	\left|
	\begin{array}{c|cccc}
		sz &  & a_1z & \cdots &a_{2g}z\\ \hline
    	a_1z &  & \multicolumn{3}{c}{\multirow{3}{*}{\raisebox{-7mm}{\scalebox{1}										{$t^{1/2}W - t^{-1/2}W^{\top}$}}}} \\
    	\raisebox{2pt}{\vdots} & & & &\\
    	a_{2g}z &  & & 
	\end{array}
	\right|
	\mp t^{-1/2}
	\left|
	\begin{array}{c|cccc}
		\mp t^{-\frac{1}{2}} &  & 0 &  \cdots & 0 \\ \hline
		a_{1}z  & & \multicolumn{3}{c}{\multirow{3}{*}{\raisebox{-7mm}{\scalebox{1}										{$t^{1/2}W - t^{-1/2}W^{\top}$}}}} \\
    	\raisebox{2pt}{\vdots}  & & & &\\
    	a_{2g}z & & & & 
	\end{array}
	\right| \\*
	\end{align*}
	
	\begin{align*}
	&=z^3\;
	\left|
	\begin{array}{c|cccc}
		sz^{-1} &  & a_1 & \cdots &a_{2g}\\ \hline
    	a_1 &  & \multicolumn{3}{c}{\multirow{3}{*}{\raisebox{-7mm}{\scalebox{1}										{$t^{1/2}W - t^{-1/2}W^{\top}$}}}} \\
    	\raisebox{2pt}{\vdots} & & & &\\
    	a_{2g} & & & & 
	\end{array}
	\right|
	\;+\; \det(t^{1/2}W - t^{-1/2}W^{\top}) \\
	& \\
	&=s\,z^2 \cdot \det(t^{1/2}W - t^{-1/2}W^{\top}) + z^3 \cdot f(t) + 			\Delta_{\ell_{1} \subset \Sigma}(t)
	\end{align*}	
for some function $f(t)$, which is a polynomial in $t^{1/2}$ and $t^{-1/2}$. Therefore, 
\begin{equation*}
\Delta_{\ell_{1} \subset \Sigma - \ell_{2}}(t) = sz^{2}\Delta_{\ell_{1} \subset \Sigma}(t) + z^{3}f(t) + \Delta_{\ell_{1} \subset \Sigma}(t).
\end{equation*}
Next we differentiate twice and set $t=1$. An easy calculation taking into account that $z(1)=0$ and $z'(1) = 1$ leads to the formula
\begin{equation*}
\Delta^{''}_{\ell_{1} \subset \Sigma - \ell_{2}}(1) = 2s\cdot\Delta_{\ell_{1} \subset \Sigma }(1)+	\Delta^{''}_{\ell_{1} \subset \Sigma}(1) = 2s + \Delta^{''}_{\ell_1 \subset\Sigma}(1).
\end{equation*}
Therefore,
\begin{equation*}
	\I(M,P)) = -2s(\ell_1 \cup \ell_2 \subset \Sigma) = 2\cdot \LE(M).
\end{equation*}

The above calculation holds for a specific bundle which has $w_2(\ad(P)) = (1,1)$. We will now prove the result for admissible bundles with $w_2(\ad(P))$ = $(0,1)$ and $w_2(\ad(P))$ = $(1,0)$ using the fact that $\LE(M) = - |\Tor(H_1(M))|\;\cdot$ \SL is an invariant of the manifold $M$ and therefore is independent of the surgery presentation. 

Without loss of generality, let us assume that $w_2(\ad(P)) = (1,0)$. After sliding $\ell_1$ over $\ell_2$, we will obtain a new surgery presentation for $M$, namely, $M = \Sigma + 0 \cdot \ell_1 + 0 \cdot \ell_{\sharp}$, where $\ell_{\sharp}$ is the new knot obtained by sliding $\ell_1$ over $\ell_2$. Note that $\ell_{\sharp}$ bounds a Seifert surface which is a band sum of $F_1$ and $F_2$ and also that $\lk_{\Sigma}(\ell_1,\ell_{\sharp})$ = 0. In the new basis, $w_2(\ad(P)) =(1,1)$, hence $\I(M,P)) = s(\ell_1 \cup \ell_{\sharp} \subset \Sigma)$ by the above argument. The independence of surgery presentation then implies that $s(\ell_1 \cup \ell_{\sharp} \subset \Sigma)$ = \SL and therefore $\I(M,P))$ is independent of the choice of admissible bundle. 
\end{proof}

\begin{example}
Given a two component link $\ML = \ell_1 \cup\,\ell_2$ in an integral homology sphere $\Sigma$, Harper and Saveliev \cite{HS} defined its Floer homology $I_{*}(\Sigma, \ML)$ as follows. The link exterior $X = \Sigma \setminus \operatorname{int} \;N(\ML)$ is a compact manifold whose boundary consists of two 2-tori. Furl it up by gluing the boundary tori together via an orientation reversing diffeomorphism $\varphi: T^2 \longrightarrow T^2$ chosen so that the resulting closed manifold $X_{\varphi}$ has homology of $S^1 \times S^2$. Then $I_{*}(\Sigma, \ML)$ = $I_{*}(X_{\varphi})$. 

The manifold $X_{\varphi}$ has the following surgery description, see \cite{HS}. Attach a band from one component of $\ML$ to the other, and call the resulting knot $k_1$. Let $k_2$ be a small unknotted circle going once around the band with linking number zero. Then $X_{\varphi}$ is the manifold obtained from $\Sigma$ by performing surgery on the link $k_1 \cup\, k_2$, with $k_1$ framed by $\pm 1$ and $k_2$ framed by $0$. A quick argument with Floer exact triangle shows that $I_{*}(X_{\varphi}) = I_{*}(Y,P)$, where $Y$ is the manifold obtained by surgery on the link $k_1\cup\, k_2$ with both components framed by zero. Since the link $k_1 \cup\, k_2$ is algebraically split, it follows from equations (\ref{eq:LEb2}) and (\ref{eq:LEXb2}) that 
\begin{equation*}
\chi (I_{*}(\Sigma, \ML)) = \chi(I_{*}(Y,P)) = -2 s(k_1 \cup k_2).
\end{equation*} 
A straightforward calculation then shows that $s(k_1 \cup k_2) = \pm \lk_{\Sigma}\,(\ell_1, \ell_2)$, which recovers the formula $\chi (I_{*}(\Sigma, \ML))  = \pm 2 \cdot \lk_{\Sigma}\,(\ell_1, \ell_2)$ of \cite{HS}. 
\end{example}


Before we move on to the case when $H_1(M)$ has non-trivial torsion, we will establish a fact which we will need for the calculations. Recall that if $\Sigma$ is a rational homology sphere and $k$ is an arbitrary knot in $\Sigma$, then there exists an integer $q$ such that $q \cdot k$ represents a zero in $H_1(\Sigma)$. For any other knot $\ell$ in $\Sigma$, the linking number $\lk_{\Sigma} (q\cdot k,\ell)$ is defined as the intersection number of a Seifert surface of $q\cdot k$ with $\ell$, and one further defines
\begin{equation*}
\lk_{\Sigma}(k,\ell) = \dfrac 1 q\cdot \lk_{\Sigma} (q\cdot k,\ell)\;\in\; \mathbb Q.
\end{equation*}
\begin{lemma}{\label{Lem:Hos}}
Let $\Sigma$ be a rational homology sphere, $\ell_2 \subset \Sigma$ a null-homologous knot and $k_1, k_2$  knots in $\Sigma \smallsetminus N(\ell_2)$. Then,  
\begin{equation}
\label{Eqn}
\lk_{\Sigma - \ell_2} (k_1,k_2) =\lk_{\Sigma} (k_1, k_2) + \lk_{\Sigma} (k_1, \ell_2) \cdot \lk_{\Sigma} (k_2,\ell_2).
\end{equation}
\end{lemma}

\begin{proof}
If $k_1, k_2$ are null-homologous in $\Sigma \smallsetminus N(\ell_2)$, the proof proceeds exactly as in \cite[Lemma 1.2]{HC}. If $k_1$ and $k_2$ are arbitrary knots in $\Sigma \smallsetminus N(\ell_2)$, there exist non-zero integers $q_1, q_2$ such that $q_1 \cdot k_1$ and $q_2 \cdot k_2$ are null-homologous in $\Sigma \smallsetminus N(\ell_2)$ and therefore in $\Sigma - \ell_2$. Then 
\begin{equation*}
\lk_{\Sigma - \ell_2}(q_1 \cdot k_1, q_2 \cdot k_2) = \lk_{\Sigma}(q_1 \cdot k_1, q_2 \cdot k_2) + \lk_{\Sigma}(q_1 \cdot k_1, \ell_2) \cdot \lk_{\Sigma}(q_2 \cdot k_2, \ell_2)
\end{equation*}
and the result follows by dividing both sides by $q_1 \cdot q_2$.
\end{proof}

\begin{prop}{\label{thm2:2}}
Let $M = \Sigma + 0 \cdot \ell_{1} + 0 \cdot \ell_2$, where $\Sigma$ is a rational homology sphere and $\ell_1, \ell_2 \subset \Sigma$ are null-homologous knots such that $\lk_{\Sigma}(\ell_1,\ell_2) = 0$. Then
\begin{equation*}
\I(M,P)) = \displaystyle \frac{2 \cdot \LE(M)}{|\Tor(H_1(M))|}
\end{equation*}

\smallskip\noindent
for any admissible bundle $\ad(P)$ over $M$ whose restriction to the exterior of $\ell_1 \cup \ell_2 \subset \Sigma$ is trivial.  
\end{prop}

\begin{proof}
Let $P$ be a $U(2)$-bundle over M such that $w_2(\ad(P))$ evaluates non-trivially on both homology classes obtained by capping off Seifert surfaces of $\ell_1$ and $\ell_2$ by meridional disks of the surgery. By the Floer exact triangle, 
\begin{align*}
	\I(\Sigma + 0 \cdot \ell_1 + 0 \cdot \ell_2))
	&= \I((\Sigma + 0 \cdot \ell_{1})- \ell_{2}))- \I(\Sigma + 0 \cdot \ell_{1}))\\
	&= -\Delta_{\ell_1 \subset \Sigma - \ell_2}^{''}(1) + \Delta_{\ell_1 \subset \Sigma}^{''}(1). 
\end{align*}
We proceed as in the proof of Proposition \ref{thm2.1} and choose Seifert surfaces $F_1$ and $F_2$ which are either disjoint or intersect in a single ribbon, and construct Seifert surfaces $F^{'}_{1}$ and $F_2$ which intersect in a circle $c$. Complete $c$ to a basis $\{ e_{1}, c, e_{3},...,e_{2g+2} \}$ of $H_{1}(F_{1}^{'})$, where $e_1$ is a meridional curve of $\ell_2$ and $\{ e_{3},...,e_{2g+2}\}$ is a basis for  $H_{1}(F_{1})$. Again,

\begin{equation*}
E\;=\;
	\begin{bmatrix}	
		\lk_{\Sigma}(e_{1},\ell_{2}) \\
		\lk_{\Sigma}(c,\ell_{2})\\
		\vdots \\
		\lk_{\Sigma}(e_{2g+2},\ell_{2})\\
	\end{bmatrix}
	\; = \;
	\begin{bmatrix}
		\pm 1 \\
		0\\
		\vdots\\
		0\\
	\end{bmatrix}
\end{equation*}

\bigskip
We would like to determine the Seifert matrix $V(\ell_1 \subset \Sigma)$ with respect to the basis $\{ e_{1}, c, e_{3},...,e_{2g+2} \}$ of $H_1(F^{'}_1)$. As $e_1$ bounds a meridional disk $D$, it is null-homologous in $\Sigma$. Since $e_1^{+}$ is disjoint from $D$, we have $\lk_{\Sigma}(e_1, e_1^{+}) = 0$. As in the torsion free case, $\lk_{\Sigma}(c,e_1^{+}) = 0$ and $\lk_{\Sigma}(e_1,c^+) = \pm 1$. For $m \geq 3$, the disk $D$ can be picked to be disjoint from $e_m^{+}$ hence $\lk_{\Sigma}(e_1,e_1^{+}) =0$ and similarly $\lk_{\Sigma}(e_m,e_1^{+}) = 0$. By definition, $\lk_{\Sigma}(c,c^+) =s$. 
Finally,
$\lk_{\Sigma}(c^+, e_m) = \lk_{\Sigma}(c, e_m^+)$ using the fact that $V(\ell_{1} \subset \Sigma) - V(\ell_{1} \subset \Sigma)^{\top} =  - I$, where I is the intersection form on $H_1(F^{'}_1; \mathbb{Q})$.
For $m \geq 3$ denote $\lk_{\Sigma}(c^+, e_m) =  \lk_{\Sigma}(c, e_m^+) = a_{m-2}$. Then

\begin{center}
	$V(\ell_1 \subset \Sigma)=$
	$\left[ 
	\begin{array}{cc|cccc}
	0 & 0 &  & 0 &  \cdots & 0 \\
	\pm 1 & s &  & a_1 & \cdots &a_{2g}\\ \hline
   0 & a_1 & & \multicolumn{3}{c}{\multirow{3}{*}{\raisebox{-7mm}{\scalebox{1}{$W$}}}} \\
   \raisebox{2pt}{\vdots} & \vdots & & & &\\
   0 &a_{2g} & & & &
	\end{array}
	\right]$
\end{center}

\medskip\noindent
Because of Lemma \ref{Lem:Hos}, the Hoste formula (\ref{eqn:Hoste}) still holds. Therefore, the rest of the proof works out as in Proposition \ref{thm2.1} to show that 
\begin{equation*}
\I(M,P)) = -2s(\ell_1\cup \ell_2 \subset \Sigma) = \displaystyle \frac {2 \cdot \LE(M)}{|\Tor(H_1(M))|}
\end{equation*}
and that this is independent of the choice of admissible bundle which restricts to a trivial bundle on the exterior of $\ell_1\cup\ell_2 \subset \Sigma$.
\end{proof}

\begin{remark}
The only remaining ambiguity in choosing an admissible bundle $\ad(P)$ over $M$ has to do with the group $\operatorname{Ext}(H_1(M),\Z_2)$. If $H_1(M)$ has no $2$-torsion, this group vanishes and Proposition \ref{thm2:2} holds for an arbitrary admissible bundle.
\end{remark}


\section{Case of $b_1(M) = 3$}
Let $\ML = \ell_1 \cup \ell_2 \cup \ell_3$ be an algebraically split oriented three-component link in a rational homology sphere $\Sigma$ such that each component of $\ML$ is null-homologous. Let $F_1, F_2, F_3$ be Seifert surfaces of the knots $\ell_1, \ell_2, \ell_3$, respectively, chosen so that $F_i \cap \ell_j = \emptyset$ for $i \ne j$. Define the Milnor triple linking number $\mu(\ell_1,\ell_2,\ell_3)$ as a signed count of points in the intersection $F_1\cap F_2\cap F_3$.

Given a closed oriented connected $3$-manifold with $b_1(M)$ = $3$, by Lemma \ref{Lem:Les}, there exists an algebraically split link $\ML = \ell_1 \cup \ell_2 \cup \ell_3$ in a rational homology sphere $\Sigma$ such that $M = \Sigma + 0\cdot \ell_1 + 0 \cdot \ell_2 + 0 \cdot \ell_3$ and the compo\-nents of $\ML$ are all null-homologous. According to Lescop \cite{LE}, 
\begin{equation} \label{{eq:LEb3}}
\LE(M) = |\Tor(H_1(M))|\cdot (\mu(\ell_1,\ell_2,\ell_3))^2.
\end{equation}
Again, $\LE(M)$ is independent of the choice of surgery presentation as above.
 
\begin{prop}{\label{thm3}}
Let $M = \Sigma + 0 \cdot \ell_{1} + 0 \cdot \ell_{2} + 0 \cdot \ell_{3}$, where $\Sigma$ is a rational homology sphere, $\ell_1, \ell_2, \ell_3 \subset \Sigma$ are null-homologous knots such that $\lk_{\Sigma}(\ell_{i}, \ell_{j})$ =0 for $i \neq j$. Then 
\begin{equation}\label{eq:chib3}
\I(M,P)) = \displaystyle - \frac{2\cdot \LE(M)}{|\Tor(H_1(M))|}
\end{equation}
for any admissible bundle $\ad(P)$ over $M$ whose restriction to the exterior of $\ell_1 \cup \ell_2 \cup \ell_3 \subset \Sigma$ is trivial. 
\end{prop}

\begin{proof}
Let $P$ be a $U(2)$-bundle over $M$ such that $w_2(\ad(P))$ evaluates non-trivially on all three homology classes obtained by capping off Seifert surfaces of $\ell_1$, $\ell_2$ and $\ell_3$ by meridional disks of the surgery. By the Floer exact triangle, 
\begin{align*}
\I(M,P)))
	&= \I((\Sigma + 0 \cdot \ell_{1}+0 \cdot \ell_{2})- \ell_{3}))- \I(\Sigma + 0 \cdot \ell_{1} + 0 \cdot \ell_{2}))\\
	&= -2s(\ell_{1} \cup \ell_{2} \subset (\Sigma - \ell_{3})) + 2s(\ell_{1} \cup \ell_{2} \subset \Sigma)\\
	&= - 2 \cdot \lk_{\Sigma - \ell_3}(c,c^{+}) + 2 \cdot \lk_{\Sigma}(c,c^{+}),
\end{align*}
where $c$ is the intersection circle $c = F_1 \cap F_2$ of the Seifert surfaces $F_1$ and $F_2$ chosen to intersect in a circle. We wish to identify the right hand side of this equation with $-2 (\mu(\ell_1,\ell_2,\ell_3))^2$. By (\ref{Eqn}), 
\[
\lk_{\Sigma - \ell_3}(c,c^{+}) - \lk_{\Sigma}(c,c^{+}) = \lk_{\Sigma}(c,\ell_3) \cdot \lk_{\Sigma}(c^+, \ell_3),
\]
therefore, 
\begin{equation*}
\I(M,P))) = - 2 \cdot \lk_{\Sigma}(c,\ell_{3}) \cdot \lk_{\Sigma}(c^{+},\ell_{3})\\ = - 2\cdot \lk(c,\ell_{3})^{2}.
\end{equation*}
Since $\lk_{\Sigma}(c,\ell_{3})$= $c \cdot F_{3}$ = $\mu(\ell_{1},\ell_{2}, \ell_{3})$, the result follows. The independence of admissible bundle follows by the change of basis argument as before.
\end{proof}

\begin{remark}
If $H_1(M)$ has no $2$-torsion, by the same reasoning as in the remark after Proposition \ref{thm2.1}, the statement of Proposition \ref{thm3} holds for any admissible bundle over $M$. 
\end{remark}

It is worth mentioning that Ruberman and Saveliev \cite{RSI} showed that $1/2\cdot \I(M,P)) = \LE(M) \mod{2}$ for all $M$ with $H_{*}(M)  = H_{*}(T^3)$ using techniques different from ours. 

\begin{example}
Given a knot $\ell \subset S^3$, Kronheimer and Mrowka \cite{KMK} defined its reduced singular instanton knot homology $\In (\ell)$ as follows. Take the knot exterior $S^3 \smallsetminus N(\ell)$ and construct a closed $3$-manifold $Y$ by attaching, along the boundary, the manifold $F \times S^2$, where $F$ is a punctured $2$-torus. The attaching is done so that the curve $\partial F \times \{ \pt \} $ matches the canonical longitude of $\ell$, and the curve $\{ \pt \} \times S^1$ matches its meridian. Then 
\begin{equation}\label{eq:KM}
I_{*}(Y,P) = \In (\ell) \oplus \In (\ell)
\end{equation}
for a particular choice of admissible bundle $P$ over $Y$, see \cite[Proposition 5.7]{KMK}. One can easily see that $Y$ is homeomorphic to the manifold obtained by $(0,0,0)$-surgery on $\ell\,\# \Br$, where $\Br$ stands for the Borromean rings. Since the link $\ell\, \# \, \Br$ is algebraically split with $\mu(\ell\, \# \Br) = 1$, it follows from equations \eqref{{eq:LEb3}}, \eqref{eq:chib3} and \eqref{eq:KM} that 
\[
\chi (\In (\ell))\;=\; \frac{1}{2} \cdot \I(Y,P)) = -1. 
\]
\end{example}

\section{Case of $b_1(M) \geq 4$}
For all closed oriented connected $3$-manifolds $M$ with $b_1(M) = n \geq 4$, the Lescop invariant $\LE (M)$ is known to vanish. 

\begin{prop}{\label{thm4}}
Let $M = \Sigma + 0 \cdot \ell_{1} + \ldots  + 0 \cdot \ell_{n}$, where $n \geq 4$, $\Sigma$ is a rational homology sphere, $\ell_1, \cdots, \ell_n \subset \Sigma$ are null-homologous knots such that $\lk_{\Sigma}(\ell_{i}, \ell_{j})$ =0 for $i \neq j$. Then $\I(M,P))=0$ for any admissible bundle $\ad(P)$ over $M$ whose restriction to the exterior of $\ML = \ell_1 \cup \cdots \cup \ell_n$ in $\Sigma$ is trivial. 
\end{prop}

\begin{proof}
Let $P$ denote the $U(2)$-bundle over $M$ such that $w_2(\ad(P))$ evaluates non-trivially on the homology classes obtained by capping off Seifert surfaces of the components $\ell_i$, $i = 1,\ldots,n$, by meridional disks of the surgery. By the Floer exact triangle, if $n = 4$,
\begin{multline}\notag
\I(M,P)) = \I((\Sigma + 0 \cdot \ell_{1}+0 \cdot \ell_{2}+ 0 \cdot \ell_{3})- \ell_{4})) \\ \qquad\qquad\qquad
 - \I(\Sigma + 0 \cdot \ell_{1} + 0\cdot \ell_{2}+ 0 \cdot \ell_{3}))\\
 =-2 \cdot \lk_{\Sigma - \ell_4}(c,\ell_{3})^2 + 2 \cdot\lk_{\Sigma}(c,\ell_{3})^2.
\end{multline}
The result now follows because  
\begin{equation*}
\lk_{\Sigma  -\ell_4}(c,\ell_{3})- \lk_{\Sigma}(c,\ell_{3}) =\lk_{\Sigma}(c,\ell_{4})\cdot \lk_{\Sigma}(\ell_{3}, \ell_{4}) = 0,
\end{equation*}
using equation (\ref{Eqn}) and the fact that $\lk(\ell_3, \ell_4) = 0$.

If $n \geq 5$, choose an admissible bundle $\ad(P)$ whose restriction to the exterior of $\ML\subset\Sigma$ is trivial, and proceed by induction. Suppose that $\I(M',P)) = 0$ for all $M'$ with $b_1(M')$ = $k$, where $4 \leq k \leq n-1$. Let $M = \Sigma + 0 \cdot \ell_{1} + \ldots + 0 \cdot \ell_n$, where $\ell_i$ are null-homologous knots in a rational homology sphere $\Sigma$ such that $\lk_{\Sigma}(\ell_{i}, \ell_{j})$ =0 for $i \neq j$. Then by the Floer exact triangle 
\begin{multline}\notag
\I(M,P)) = \I((\Sigma + 0 \cdot \ell_{1}+\ldots + 0 \cdot \ell_{n-1})- \ell_{n}))\\ -  \I(\Sigma + 0 \cdot \ell_{1} 	+ \ldots + 0 \cdot \ell_{n-1})).
\end{multline}
Since 
\begin{equation*}
b_1((\Sigma - \ell_{n})+ 0 \cdot \ell_{1}+\ldots + 0 \cdot \ell_{n-1}) = b_1(\Sigma + 0 \cdot \ell_{1} + \ldots + 0 \cdot \ell_{n-1}) = n-1,
\end{equation*}
we conclude that $\I(M,P)) = 0$.
\end{proof}

\begin{remark}
As before, the result holds for all admissible bundles over $M$ if $H_1(M)$ has no $2$-torsion; the latter restriction can conjecturally be removed. 
\end{remark}


\section{An example}
The following example illustrates the appearance of the factor $|\Tor H_1(M)|$ in the Lescop invariant from a gauge-theoretic viewpoint.

Let $Y$ be a closed oriented 3-manifold with torsion-free first homology of rank at least one, and consider the manifold $M = Y\;\#\;L(p,q)$ obtained by connect summing $Y$ with a lens space $L(p,q)$. It follows from the connected sum formula in Lescop \cite{LE} that 
\[
\lambda_L (M) = p\cdot \lambda_L (Y) = |\Tor H_1 (M)|\cdot \lambda_L (Y)
\]
but we wish to explain the factor $|\Tor H_1 (M)|$ from a gauge-theoretic viewpoint.

Let $P$ be an admissible bundle over $M$ obtained by matching an admissible bundle $Q$ over $Y$ with a trivial bundle over $L(p,q)$. As in \cite[Section 3.2]{RSI}, the holonomy map provides a bijective correspondence between gauge equivalence classes of projectively flat connections in $P$ and conjugacy classes of projective representations
\begin{equation*}
\alpha: \pi_1(Y\;\# \; L(p,q)) \to SU(2)
\end{equation*}
with the Stiefel--Whitney class $w_2(P)$. Since $\pi_1 (Y\,\#\,L(p,q)) = \pi_1(Y) * \pi_1(L(p,q))$ is a free product, all such $\alpha$ will be of the form $\alpha = \beta * \gamma$, where $ \beta: \pi_1(Y) \to SU(2)$ is a projective representation with the Stiefel--Whitney class $w_2(Q)$, and $\gamma: \pi_1(L(p,q)) \to SU(2)$ is a representation. We will assume for the sake of simplicity that the character variety of $\pi_1(Y)$ is non-degenerate; the general case can be handled using perturbations. Note that since $\beta$ is irreducible, each pair of conjugacy classes $[\beta], [\gamma]$ gives rise to a family of $[\alpha]$ parameterized by $SU(2)/\Stab(\gamma)$.

We will next examine the $SU(2)$-character variety of $\pi_1(L(p,q))$. Since $\pi_1 (L(p,q)) = \Z_p$ is abelian, one may assume after conjugation that the image of a representation $\gamma: \pi_1(L(p,q)) \to SU(2)$ is a unit complex number. Fix a generator $1 \in \pi_1 (L(p,q))$, then such representations $\gamma$ correspond to the roots of unity $\gamma(1) = \exp(2\pi in/p)$, with $0 \le n \le p-1$. The number of conjugacy classes of $\gamma$ and the size of $\Stab(\gamma)$ depend on the parity of $p$ as follows.

When $p$ is odd, the trivial representation $\theta$ corresponding to $n = 0$ is the only central representation. Its stabilizer equals $SU(2)$ hence it gives rise to a single point in the character variety of $\pi_1 (M)$ for each $[\beta]$. Other representations $\gamma$ are non-central and, since $\cos (2 \pi n/p) = \cos(2 \pi (p-n)/p)$, there are $(p-1)/2$ conjugacy classes of them enumerated by $\tr(\gamma(1)) = 2 \cos (2 \pi n/p)$, $1 \le n \le (p-1)/2$. Since each of these $\gamma$ has stabilizer $U(1)$, it gives rise to a copy of $SU(2)/U(1) = S^2$ in the character variety of $\pi_1 (M)$ for each $[\beta]$.

When $p$ is even, there are two central representations, $\pm \theta$, each giving rise to two points in the character variety of $\pi_1 (M)$ for each $[\beta]$. Like in the odd case, each of the remaining $(p-2)/2$ conjugacy classes corresponding to $\tr(\gamma(1)) = 2 \cos (2 \pi n/p)$ gives rise to a 2-sphere's worth of representations in the character variety of $\pi_1 (M)$ for each $[\beta]$.

One can easily see that the 2-spheres in the character variety of $\pi_1 (M)$ described above are non-degenerate in the Morse--Bott sense. Therefore, each of them contributes $\pm \chi (S^2) = \pm 2$ to the Euler characteristic of the instanton Floer homology of $M$. The latter follows for instance from \cite[Theorem 5.1]{SH} which compares the Wilson loop perturbations of Floer \cite{F1} to Morse-type perturbations. The signs of these contributions can be figured out by either computing the Floer indices of $\beta*\gamma$ or by using an ad hoc argument equating the Euler characteristic of instanton homology to the Lescop invariant. 

The final outcome is that each $[\beta]$ in the character variety of $\pi_1 (Y)$ contributes $1 + 2\cdot (p-1)/2 = p$ to the Euler characteristic $\I(M,P)$ if $p$ is odd, and it contributes $2 + 2\cdot (p-2)/2 = p$ if $p$ is even. In both cases, this results in the desired formula $\I(M,P)) = p\cdot \I(Y,Q)) = |\Tor H_1(M)| \cdot \I(Y,Q))$.



\begin{thebibliography}{999}

\bibitem{AM}
S.~Akbulut, J.~McCarthy, 
\textit{Casson's Invariant for Oriented Homology 3-Spheres: An Exposition.} Princeton Univ. Press, 1990.

\bibitem{BD}
P.~Braam, S.~Donaldson, 
\textit{Floer's work on instanton homology, knots, and surgery}. In:
The Floer memorial volume, 195--256, Progr. Math. \textbf{133}, 
Birkh\"auser, 1995.

\bibitem{TC}
T.~Cochran, 
\textit{Derivatives of links: Milnor’s concordance invariants and Massey’s products}, Mem. Amer. Math. Soc. \textbf{84} (1990), no. 427.

\bibitem{DO}
S.~Donaldson, \textit{Floer homology groups in Yang-Mills theory.} Cambridge
Univ. Press, 2002.

\bibitem{F1}
A. Floer, \textit{An instanton-invariant for 3-manifolds}, Comm. Math. Phys. \textbf{118} (1988), 215--240


\bibitem{F2}
A. Floer, \textit{Instanton homology, surgery and knots}, In: Geometry of low-dimensional manifolds, 1 (Durham, 1989). 97--114, London Math. Soc. Lecture Note Ser., \textbf{150}, Cambridge Uni. Press, 1990

\bibitem{HS} 
E. Harper, N. Saveliev, \textit{Instanton Floer homology for two-component links}, J. Knot Theory Ramifications \textbf{21} (2012), 1250054 (8 pages)


\bibitem{HC}
J. Hoste, \textit{A formula for Casson’s invariant}, Trans. Amer. Math. Soc. \textbf{297} (1986), 547--562

\bibitem{KMK}
P. Kronheimer, T. Mrowka, \textit{Khovanov homology is an unknot-detector}, 
Publ. Math. Inst. Hautes Etudes Sci. \textbf{113} (2011), 97--208

\bibitem{KMS}
P. Kronheimer and T. Mrowka, \textit{Knots, sutures, and excision}, J. Differential Geom. \textbf{84} (2010), 301--364

\bibitem{LE}
C. Lescop, \textit{Global Surgery Formula for Casson–Walker Invariant}, Annals of Mathematics Studies, Vol. 140. Princeton Univ. Press, 1996.

\bibitem{MM}
M. Mari\~no and G. Moore, \textit{3-manifold topology and the Donaldson-Witten partition function}, Nucl. Phys. \textbf{B547} (1999), 569--598

\bibitem{MT}
K. Masataka, \textit{Casson’s knot invariant and gauge theory}, Topology Appl. \textbf{112} (2001), 111--135

\bibitem{RSI}
D. Ruberman and N. Saveliev, \textit{Rohlin’s invariant and gauge theory. I. Homology
3-tori}, Comment. Math. Helv. \textbf{79} (2004), 618–-646

\bibitem{SH}
N. Saveliev, \textit{Invariants for Homology 3-Spheres}. Springer, 2002.

\bibitem{CT}
C. Taubes, \textit{Casson's Invariant and Gauge Theory},  J. Differential Geom. \textbf{31} (1990), 547--599

\end{thebibliography}
\end{document}